\documentclass[a4paper,11pt]{amsart}

\usepackage{graphicx}

\usepackage{amsmath,amsfonts,fixltx2e}
\usepackage{amssymb,epsfig}
\usepackage{longtable}
\usepackage{srcltx}
\usepackage{calc}
\usepackage[font=small,labelfont=bf]{caption}
\newtheorem{theorem}{Theorem}[section]
\newtheorem{lemma}[theorem]{Lemma}
\newtheorem{proposition}[theorem]{Proposition}

\theoremstyle{definition}
\newtheorem{definition}[theorem]{Definition}
\newtheorem{example}[theorem]{Example}

\theoremstyle{remark}
\newtheorem{remark}[theorem]{Remark}

\DeclareMathOperator{\Pre}{Pref}
\DeclareMathOperator{\Pkh}{PV}
\DeclareMathOperator{\Ff}{Fact}
\newcommand{\pref}[1]{\Pre ( #1 )}

\newcommand{\pv}[1]{\Pkh ( #1 )}
\newcommand{\pvn}[1]{\Pkh_{m} ( #1 )}

\newcommand{\Nn}{\mathbb N}
\newcommand{\zz}{\mathbb Z}
\newcommand{\Pp}{\mathcal P}
\newcommand{\Aa}{\mathcal A}
\newcommand{\vv}[1]{\bar{\mathbf{#1}}}

\numberwithin{equation}{section}

\begin{document}

\title[Aperiodic pseudorandom number generators]{Aperiodic pseudorandom number generators based on infinite words}


\author[\v{L}. Balkov\'a]{\v{L}ubom\'ira Balkov\'a}
\address{Department of Mathematics, FNSPE, Czech Technical University in Prague,
Trojanova 13, 120 00 Praha 2, Czech Republic}
\email{lubomira.balkova@gmail.com}

\author[M. Bucci]{Michelangelo Bucci}
\address{Department of Mathematics, University of Li\`ege, Grande traverse 12 (B37), B-4000 Li\`ege, Belgium}
\curraddr{}
\email{michelangelo.bucci@gmail.com}
\thanks{}

\author[A. De Luca]{Alessandro De Luca}
\address{DIETI, Universit\`a degli Studi di Napoli Federico II\\via Claudio, 21,
80125 Napoli, Italy}
\curraddr{}
\email{alessandro.deluca@unina.it}
\thanks{}

\author[J. Hladk\'y]{Ji\v{r}\'i Hladk\'y}
\address{Department of Mathematics, FNSPE, Czech Technical University in Prague,
Trojanova 13, 120 00 Praha 2, Czech Republic}
\curraddr{}
\email{hladky.jiri@gmail.com}
\thanks{}

\author[S. Puzynina]{Svetlana Puzynina}
\address{LIP, ENS Lyon, France, and Sobolev Institute of Mathematics, Russia}
\curraddr{LIP, ENS Lyon, 46 All\'ee d'Italie Lyon 69364 France }
\email{s.puzynina@gmail.com}
\thanks{}

\subjclass[2010]{68R15}

\date{}

\dedicatory{}

\begin{abstract}
In this paper we study how certain families of aperiodic infinite
words can be used to produce aperiodic pseudorandom number
generators (PRNGs) with good statistical behavior. We introduce
the \emph{well distributed occurrences} (WELLDOC) combinatorial
property for infinite words, which guarantees
absence of the lattice structure defect in related
pseudorandom number generators. An infinite word $u$ on a $d$-ary
alphabet has the WELLDOC property if, for each factor $w$ of $u$,
positive integer $m$, and vector $\mathbf v\in\zz_{m}^{d}$, there
is an occurrence of $w$ such that the Parikh vector of the prefix
of $u$ preceding such occurrence is congruent to $\mathbf v$
modulo $m$. (The Parikh vector of a~finite word $v$ over an alphabet $\mathcal A$ has its $i$-th component equal to the number of occurrences of the $i$-th letter of $\mathcal A$ in $v$.)
We prove that Sturmian words, and more generally
Arnoux-Rauzy words and some morphic images of them, have the WELLDOC
property. Using the TestU01~\cite{TestU01} and
PractRand~\cite{PractRand} statistical tests, we moreover show
that not only the lattice structure is absent, but also other
important properties of PRNGs are improved when linear
congruential generators are combined using infinite words having
the WELLDOC property.
\end{abstract}

\maketitle




%
%
%
%
%
%
%

\section*{Introduction}
Pseudorandom number generators aim to produce random numbers using
a deterministic process. No wonder they suffer from many defects.
The most usual ones -- linear congruential generators -- are known
to produce periodic sequences with a defect called the lattice
structure. Guimond et al.~\cite{GuPaPa} proved that when two
linear congruential generators are combined using infinite words
coding certain classes of quasicrystals or, equivalently, of
cut-and-project sets, the resulting sequence is aperiodic and has
no lattice structure. For some other related results concerning
aperiodic pseudorandom generators we refer to \cite{GuPa,GuPaPa1}.
We mention that although the lattice structure is considered as a
defect of a random number generator, it can be useful in some
applications for approximation of the uniform distribution
\cite{referee2}.

We have found a combinatorial condition --  \emph{well distributed
occurrences}, or WELLDOC for short -- that also guarantees absence
of the lattice structure in related pseudorandom generators. The
WELLDOC property for an infinite word $u$ over an alphabet $
{\mathcal A}$ means that for any integer $m$ and any factor $w$ of
$u$, the set of Parikh vectors modulo $m$ of prefixes of $u$
preceding the occurrences of $w$ coincides with $\mathbb Z_m^{|
{\mathcal A}|}$ (see Definition~\ref{comb_cond}). In other words,
among Parikh vectors modulo $m$ of such prefixes one has all
possible vectors.
 Besides giving generators without lattice structure, the WELLDOC property is an interesting combinatorial property of infinite words itself. 
We prove that the WELLDOC property holds for the family of Sturmian
words, and more generally for Arnoux-Rauzy words.


Sturmian words constitute a well studied family of infinite
aperiodic words. 
Let $u$ be an infinite word, i. e., an infinite sequence of
elements from a finite set called an alphabet. The {\it (factor)
complexity} function 
counts the number of distinct factors of $u$ of length $n.$ A
fundamental result of Morse and Hedlund \cite{MoHe1} states that a
word $u$ is eventually periodic if and only if for some $n$ its
complexity is less than or equal to $n$. 
Infinite words of complexity $n+1$ for all $n$ are called {\it
Sturmian words,} and hence they are aperiodic words of the
smallest complexity. The most studied Sturmian word is the
so-called Fibonacci word
\[01001010010010100101001001010010\ldots\]
fixed by the morphism $0\mapsto 01$ and $1\mapsto 0$. (See
Section~\ref{CoW} for formal definitions.) The first systematic
study of Sturmian words was given by Morse and Hedlund
in~\cite{MorHed1940}. Such sequences arise naturally in many
contexts, and admit various types of characterizations of
geometric and combinatorial nature (see, e.g., \cite{Lo}).

Arnoux-Rauzy words were introduced in \cite{ArRa} as natural
extensions of Sturmian words to multiliteral alphabets (see
Definition \ref{defAR}). Despite the fact that they were
introduced as generalizations of Sturmian words, Arnoux-Rauzy
words display a much more complex behavior. In particular, we have
two different proofs of the WELLDOC property for Sturmian words, and
only one of them can be generalized to Arnoux-Rauzy words. In the
sequel we provide both of them.

An infinite word with the WELLDOC property is then used to combine two
linear congruential generators and form an infinite aperiodic
sequence with good statistical behavior. Using the
TestU01~\cite{TestU01} and PractRand~\cite{PractRand} statistical
tests, we have moreover shown that not only the lattice structure
is absent, but also other important properties of PRNGs are
improved when linear congruential generators are combined using
infinite words having the WELLDOC property.

The paper is organized as follows. In the next section, we give
some background on pseudorandom number generation. 
Next, in Section~2, we
give the basic combinatorial definitions needed for our main
results, including the WELLDOC property, and we 
prove that the WELLDOC property of $u$ guarantees absence of the
lattice structure of the PRNG based on $u$. In Sections~3 and~4,
we prove that the property holds for Sturmian and Arnoux-Rauzy
words. Finally, in Section 5, we present results of empirical
tests of PRNGs based on words having the WELLDOC property.

A preliminary version of this paper~\cite{WDO}, using the acronym
\emph{WDO} instead of WELLDOC, was presented at the WORDS 2013
conference.

\section{Pseudorandom Number Generators and Lattice Structure}\label{CoABoIW}
For the sake of our discussion, any infinite sequence of integers
can be understood as a \emph{pseudorandom number generator
(PRNG)}; see also~\cite{GuPaPa}.
The generators the most widely used in the past -- linear congruential generators -- are known to suffer from a~defect called the lattice structure (they possess it already from dimension $2$ as shown in~\cite{Marsaglia}).

Let $Z=(Z_n)_{n \in \mathbb N}$ be a PRNG whose output is a~finite set $M \subset \mathbb N$.
We say that $Z$ has the \emph{lattice structure} if there exists $t \in \mathbb N$ such that the set $$\{(Z_i,Z_{i+1}, \dots, Z_{i+t-1})\bigm | i \in \mathbb N\}$$ is covered by a~family of parallel equidistant hyperplanes and at the same time, this family does not cover the whole lattice $$M^t=\{(A_1,A_2,\dots, A_t)\bigm | A_i \in M \ \text{for all $i \in \{1,\dots, t\}$}\}.$$

Recall that a~\emph{linear congruential generator} (LCG) $(Z_n)_{n \in \mathbb N}$ is given by parameters $a, m, c \in \mathbb N$ and defined by the recurrence relation $Z_{n+1}=aZ_n+c \mod m$.
Let us mention a~famous example of a~LCG whose lattice structure is striking.
For $t=3$, the set of triples of RANDU, i.e., $\{(Z_i, Z_{i+1}, Z_{i+2})\bigm | i \in \mathbb N\}$
is covered by only $15$ parallel equidistant hyperplanes, see Figure~\ref{RANDU}.
\begin{figure}[!h]
\begin{center}
\resizebox{\textwidth}{!}{\includegraphics{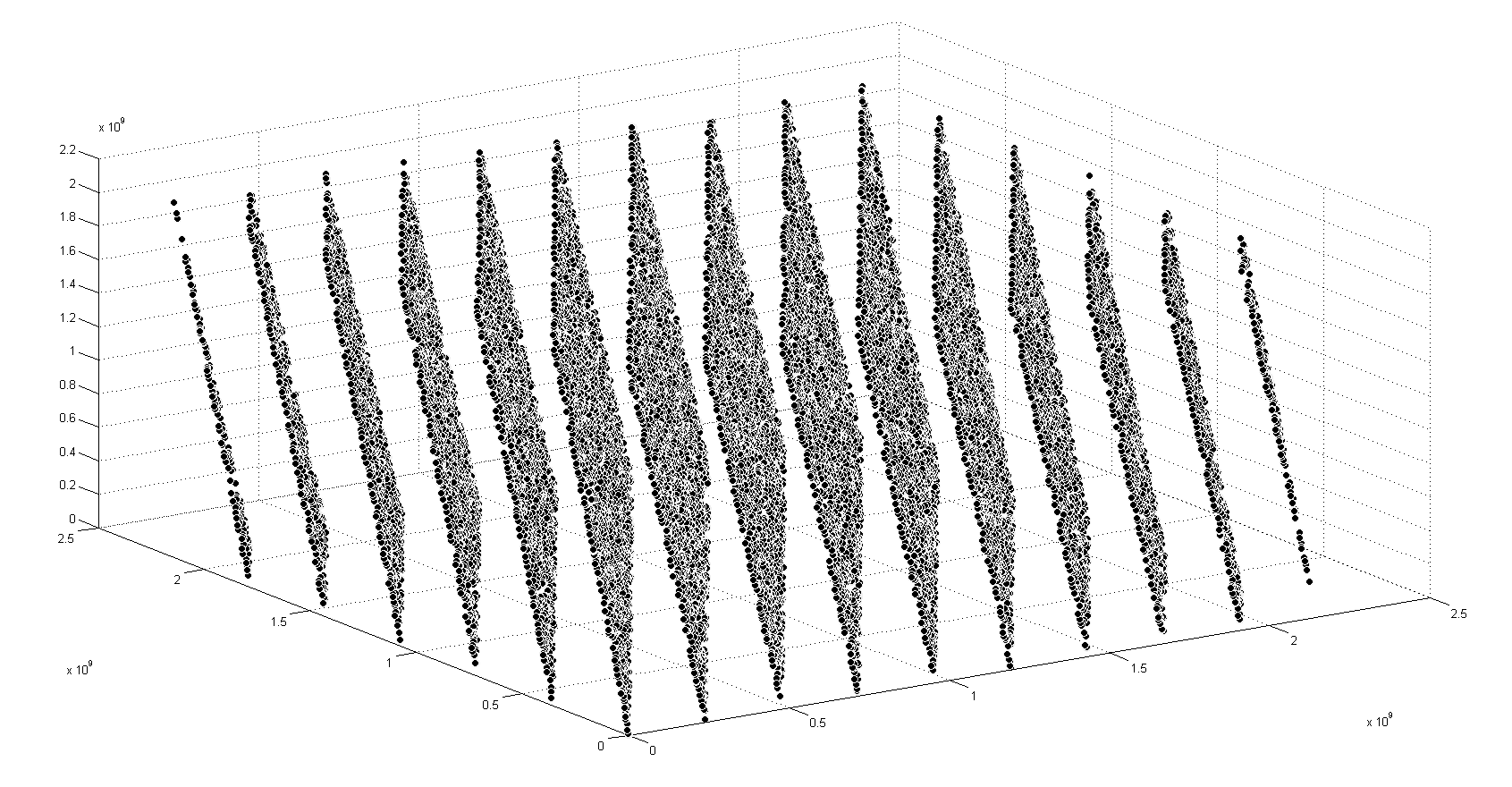}}
\caption{The triples of RANDU -- the LCG with $a=(2^{16}+3), m=2^{31}, c=0$ -- are covered by as few as $15$ parallel equidistant planes.}\label{RANDU}
\end{center}
\end{figure}

In the paper of Guimond et al.~\cite{GuPaPa}, a~restricted version of the following sufficient condition for the absence of the lattice structure is formulated.
\begin{proposition}\label{za_zb}
Let $Z$ be a~PRNG whose output is a~finite set $M \subset \mathbb N$ containing at least two elements. Assume there exists for any $A,B \in M$ and for any $\ell \in \mathbb N$  an $\ell$-tuple $(A_1,A_2,\dots, A_\ell)$ such that both $(A_1,A_2,\dots, A_\ell, A)$ and $(A_1,A_2,\dots, A_\ell,B)$ are $(\ell+1)$-tuples of the generator $Z$. Then $Z$ does not have the lattice structure.
\end{proposition}

\begin{remark}
Proposition ~\ref{za_zb} can be reformulated in terms of
combinatorics on words (see Section~\ref{CoW}) as follows: Let $Z$
be a~PRNG whose output is a~finite set $M \subset \mathbb N$
containing at least two elements. If for any $A,B\in M$ and any
length $\ell$ $Z$ has a~right special factor of length $\ell$ with
right extensions $A$ and $B$, then $Z$ does not have the lattice
structure.
\end{remark}

Since Proposition~\ref{za_zb} is formulated for a~restricted class
of generators in~\cite{GuPaPa} (see Lemma 2.3 ibidem), we will
provide its proof. However, we point out that all ideas of the
proof are taken from~\cite{GuPaPa}. We start with an auxiliary
lemma.

Let us denote $\lambda = \gcd\{A-B \mid A,B \in M \}.$

\begin{lemma}\label{skok_smer_i}
Let $Z$ be a~PRNG satisfying all assumptions of Proposition~\ref{za_zb}. Let $\vv n$ be the unit normal vector of a~family of parallel equidistant
hyperplanes covering all $t$-tuples of $Z$. Assume $\vv e_i$ (the $i$-th vector of the canonical basis of the Euclidean space $\mathbb R^t$) is not orthogonal to $\vv n$.
Then the distance $d_i$ of adjacent hyperplanes in the family along $\vv e_i$ is of the form $\lambda/k$ for some $k \in \mathbb N$.
\end{lemma}

\begin{remark}
The distance $d_i$ of adjacent hyperplanes $W_0, W_1$ along $\vv e_i$ means $|x_i-y_i|$ for any $\vv x \in W_0$ and $\vv y \in W_1$, where the $j$-th components of $\vv x$ and $\vv y$ satisfy $x_j=y_j$ for all $j \in \{1,\dots,t\}, j\not =i$.
This is a~well defined term because the hyperplanes in the family are of the form $W_j\equiv \vv x\cdot \vv n=\alpha+jd, \ j \in \mathbb Z$,
where $d$ is the distance of adjacent hyperplanes in the family and $\cdot$ denotes the standard scalar product.
Thus, without loss of generality, consider the adjacent hyperplanes
$$W_0\equiv \vv x \cdot \vv n=\alpha \quad \text{and} \quad W_1 \equiv \vv x \cdot \vv n=\alpha+d. $$
Then for any $\vv x \in W_0$ and $\vv y=\vv x+s\vv e_i$ from $W_{1}$, we have
$$\begin{array}{rcl}
\vv y \cdot \vv n&=&\alpha+d=\vv x\cdot \vv n+d,\\
\vv y\cdot \vv n&=&\vv x\cdot \vv n+s\vv e_i\cdot \vv n=\vv x\cdot \vv n+sn_i,
\end{array}$$
where $n_i$ is the $i$-th component of $\vv n$. Consequently,
$d_i=|s|=\left|\frac{d}{n_i}\right|$ and is the same for any
choice of $\vv x$ and $\vv y$ which differ only in their $i$-th
component and belong to adjacent hyperplanes.
\end{remark}
\begin{proof}[Proof of Lemma~\ref{skok_smer_i}]
Let us start with a useful observation.
Let $\vv z$ belong to a~hyperplane $W$ of the family in question.
\begin{enumerate}
\item If $\vv e_j$ is orthogonal to $\vv n$, then we may change the $j$-th component of $\vv z$ in an arbitrary way and the resulting vector will belong to the same hyperplane, i.e., if $W\equiv \vv x\cdot \vv n=\alpha$, then clearly $(\vv z+\beta\vv e_j)\cdot\vv n=\vv z\cdot\vv n=\alpha$ for any $\beta \in \mathbb R$, thus $\vv z+\beta\vv e_j$ belongs to $W$.
\item If $\vv e_j$ is not orthogonal to $\vv n$ and the distance $d_j$ of adjacent hyperplanes along $\vv e_i$ in the family is of the form $\lambda/k$ for some $k \in \mathbb N$, then $\vv z+r\lambda\vv e_j$ belongs to the family for any $r \in \mathbb Z$.
    This follows from a~repeated application of the fact that if $\vv z$ belongs to a~hyperplane $W$, then $\vv z+\frac{\lambda}{k}\vv e_j$ belongs to an adjacent hyperplane of $W$.
\end{enumerate}
Let us proceed by contradiction, i.e., we assume that there exists
$i\in \{1,\dots, t\}$ such that $\vv e_i$ is not orthogonal to
$\vv n$ and the distance along $\vv e_i$ of adjacent hyperplanes
of the family in question is not of the form $\lambda/k, \
k \in \mathbb N$. Take the largest of such indices and denote it
by $\ell$. Choose $A,B \in M$ arbitrarily. According to
assumptions, there exists an $(\ell -1)$-tuple $(A_1, A_2, \dots,
A_{\ell-1})$ such that both $(A_1, A_2, \dots, A_{\ell-1}, A)$ and
$(A_1, A_2, \dots, A_{\ell-1}, B)$ are $\ell$-tuples of $Z$. It is
therefore possible to find two $t$-tuples of $Z$ such that the
first one is of the form $(A_1, A_2, \dots, A_{\ell-1}, A,
A_{\ell+1}, \dots, A_t)$ and the second one of the form $(A_1,
A_2, \dots, A_{\ell-1}, B, {\hat A}_{\ell+1}, \dots, {\hat A}_t)$.
These two $t$-tuples -- considered as vectors in $\mathbb R^t$ --
belong by the assumption of Lemma~\ref{skok_smer_i} to some
hyperplanes in the family. Since all vectors $\vv e_j, \ j \in
\{\ell+1, \dots, t\}$ are either orthogonal to $\vv n$ or the
distance of adjacent hyperplanes along $\vv e_j$ is of the form
$\lambda/k$ for some $k \in \mathbb N$, we can change the
last $t-\ell$ coordinates ${\hat A}_{\ell+1}, \dots, {\hat A}_t$
of the second vector to arbitrary values from $M$ (we transform them
into $A_{\ell+1}, \dots, A_t$) and it will still belong to
a~hyperplane in the family. This is a~consequence of the
observation at the beginning of this proof. Hence, both vectors
$(A_1, A_2, \dots, A_{\ell-1}, A, A_{\ell+1}, \dots, A_t)$ and
$(A_1, A_2, \dots, A_{\ell-1}, B, A_{\ell+1}, \dots, A_t)$ belong
to some hyperplanes of the family. Their distance along $\vv
e_{\ell}$ equals $|A-B|$, i.e., $d_\ell$ divides $A-B$. Since $A,
B$ have been chosen arbitrarily, it follows that $d_\ell$ divides
$\lambda$, i.e., $\lambda=k d_\ell$ for some $k \in \mathbb N$,
which is a~contradiction with the choice of $\vv e_\ell$.
\end{proof}

\begin{proof}[Proof of Proposition~\ref{za_zb}]
Let $\vv n$ be the unit normal vector of a~family of parallel equidistant hyperplanes covering all $t$-tuples of $Z$.
Suppose without loss of generality that $\vv e_1, \dots, \vv e_{\ell}$ are not orthogonal to $\vv n$ and $\vv e_{\ell+1}, \dots, \vv e_t$ are orthogonal to $\vv n$.
Let $\vv z=(Z_n, Z_{n+1}, \dots, Z_{n+t-1})$ be a~$t$-tuple of $Z$, thus $\vv z$ belongs to one of the hyperplanes.
Take any vector $\vv y \in M^t$ and let us show that it belongs to a~hyperplane in the family.
\begin{enumerate}
\item Any vector from $M^t$ which differs from $\vv z$ only in the first ${\ell}$ components belongs to a~hyperplane of the family.
This comes from Lemma~\ref{skok_smer_i} because when we change for $i \in \{1,\dots, \ell\}$ the $i$-th component of $\vv z$ by $d_i=\frac{\lambda}{k}$, then we jump on the adjacent parallel hyperplane. So, any transformation of the $i$-th component of $\vv z$ into another value from $M$ means a~finite number of jumps from one hyperplane onto another. Hence, we may transform $\vv z$ so that it has the first $\ell$ components equal to $\vv y$ and the obtained vector $\vv x$ belongs to a~hyperplane in the family.
\item Any vector from $M^t$ which differs from $\vv x$ only in the last $t-{\ell}$ components belongs to the same hyperplane as $\vv x$.
This comes from the orthogonality $\vv e_i \perp \vv n$ for $i>{\ell}$ (the argument is the same as in the proof of Lemma~\ref{skok_smer_i}). Since $\vv y$ differs from $\vv x$ only in the last $t-{\ell}$ components, $\vv y$ belongs to a~hyperplane in the family.
\end{enumerate}
\end{proof}

\section{Combinatorics on Words and the WELLDOC Property}\label{CoW}

\subsection{Backgrounds on Combinatorics on Words}
In the following, ${\mathcal A}$ denotes a~finite set of symbols
called \emph{letters}; the set ${\mathcal A}$ is therefore called
an \emph{alphabet}. A~\emph{finite word} is a finite string
$w=w_1w_2\dots w_{n}$ of letters from ${\mathcal A}$; its length
is denoted by $|w| = n$ and $|w|_a$ denotes the number of
occurrences of $a \in { {\mathcal A}}$ in $w$. The empty word, a
neutral element for concatenation of finite words, is denoted
$\varepsilon$ and it is of zero length. The set of all finite
words over the alphabet ${\mathcal A}$ is denoted by $\mathcal
A^*$.

Under an \emph{infinite word} we understand an infinite sequence
${ u}=u_0u_1u_2\dots $ of letters from ${\mathcal A}$. A~finite word $w$ is
a~\emph{factor} of a~word $v$ (finite or infinite) if there exist
words $p$ and $s$ such that $v= pws$. If $p = \varepsilon$, then
$w$ is said to be a~\emph{prefix} of $v$; if $s = \varepsilon$,
then $w$ is a~\emph{suffix} of~$v$. The set of factors and
prefixes of $v$ are denoted by $\Ff(v)$ and $\pref v$,
respectively. If $v=ps$ for finite words $v,p,s$, then we write
$p=vs^{-1}$ and $s=p^{-1}v$.

An infinite word $ u$ over the alphabet $ {\mathcal A}$ is called \emph{eventually periodic} if it is of the form $ u=vw^{\omega}$,
where $v$, $w$ are finite words over $ {\mathcal A}$ and $\omega$ denotes an infinite repetition. An infinite word is called \emph{aperiodic} if it is not eventually periodic.

For any factor $w$ of an infinite word ${ u}$, every
index $i$ such that $w$ is a prefix of the infinite word
$u_iu_{i+1}u_{i+2} \dots$ is called an \emph{occurrence}
of $w$ in ${ u}$.
An infinite word $u$ is \emph{recurrent} if each of its factors has infinitely many
occurrences in $u$.

The \emph{factor complexity} of an infinite word ${ u}$ is
a~map $\mathcal{C}_u: \mathbb{N} \mapsto \mathbb{N}$ defined
by $\mathcal{C}_u(n):=\text{the number of factors of length $n$ contained in $ u$}$.
The factor complexity of eventually periodic words is bounded, while the factor complexity of an aperiodic word $ u$ satisfies $\mathcal{C}_u(n)\geq n+1$ for all $n \in \mathbb N$.
A~\emph{right extension} of a~factor $w$ of $ u$ over the alphabet $ {\mathcal A}$ is any letter $a\in {\mathcal A}$ such
that $wa$ is a~factor of $u$. Of course, any factor of
${ u}$ has at least one right extension. A~factor $w$ is
called \emph{right special} if $w$ has at least two right
extensions. Similarly, one can define a~\emph{left extension} and
a~\emph{left special} factor. A factor is \emph{bispecial} if it is both right and left
special.
An aperiodic word contains right special factors of any length.

The \emph{Parikh vector} of a finite word $w$ over an alphabet $\{0, 1, \dots, d-1\}$ is defined as $(|w|_0, |w|_1, \dots,
|w|_{d-1} )$. 
For a finite or infinite word $u= u_0 u_1 u_2 \dots$,
$\Pre_n u$ will denote the prefix of length $n$ of $u$, i.e.,
$\Pre_n u = u_0 u_1 \dots u_{n-1}$. 

In some of the examples we consider are morphic words. A
\emph{morphism} is a function $\varphi : \mathcal{A}^*\to
\mathcal{B}^*$ such that $\varphi(\varepsilon) = \varepsilon$ and
$\varphi(wv) = \varphi(w)\varphi(v)$, for all $w, v \in {\mathcal
A}^*$. Clearly, a morphism is completely defined by the images of
the letters in the domain. A morphism is \emph{prolongable} on $a
\in \mathcal{A}$, if $|\varphi(a)|\geq 2$ and
$a$ is a prefix of $\varphi (a)$. 
If $\varphi $ is prolongable on
$a$, then $\varphi^n(a)$ is a proper prefix of $\varphi^{n+1}(a)$,
for all $n \in \mathbb{N}$. Therefore, the sequence
$(\varphi^n(a))_{n\geq 0}$ of words defines an infinite word $u$
that is a fixed point of $\varphi$. Such a word $u$ is a (pure) \emph{morphic} word.

Let us introduce a~combinatorial condition on infinite words that -- as we will see later -- guarantees no lattice structure for the associated PRNGs.
\begin{definition}[The WELLDOC property]\label{comb_cond}
We say that an aperiodic infinite word $u$ over the alphabet
$\{0,1,\ldots ,d-1\}$  has \emph{well distributed occurrences} (or
has \emph{the WELLDOC property}) if for any $m \in \mathbb N$ and
any factor $w$ of $ u$ the word $ u$ satisfies the following
condition. If $i_0, i_1, \dots$ denote the occurrences of $w$ in $
u$, then
\[\left\{\left(|\Pre_{i_j} u|_0, \ldots, |\Pre_{i_j} u|_{d-1}\right)
\bmod m \mid j \in \mathbb N\right\}=\zz_m^d\,;\]
that is, the Parikh vectors of $\Pre_{i_{j}}u$ for $j\in\Nn$,
when reduced modulo $m$, give the whole set $\zz_{m}^{d}$.
\end{definition}

We define the WELLDOC property for aperiodic words since it
clearly never holds for periodic ones. It is easy to see that if a
recurrent infinite word $u$ has the WELLDOC property, then for
every vector $\mathbf v\in \zz_{m}^{d}$ there are infinitely many
values of $j$ such that the Parikh vector of $\Pre_{i_j}u$ is
congruent to $\mathbf v$ modulo~$m$.

\begin{example} The Thue-Morse word $$u=01101001100101101001011001101001\cdots,$$ which is a fixed point of the morphism $0\mapsto 01$, $1 \mapsto 10$, does not satisfy the WELLDOC property. Indeed, take $m=2$ and $w=00$, then $w$ occurs only in odd positions $i_j$ so
that $(|\Pre_{i_j}u|_0+|\Pre_{i_j}u|_1) = i_j$ is odd. Thus, e.g.,
$$(|\Pre_{i_j}u|_0, |\Pre_{i_j}u|_1)\allowbreak
 \bmod 2\neq (0,0),$$ and hence
$$\{(|\Pre_{i_j}u|_0, |\Pre_{i_j}u|_1) \bmod 2 \mid j \in \mathbb N\} \neq
\zz_2^2.$$
\end{example} 

\begin{example} We say that an infinite word $u$ over an alphabet
${\mathcal A}$, $|{\mathcal A}|=d$, is \emph{universal} if it contains all finite words
over ${\mathcal A}$ as its factors. It is easy to see that any universal word
satisfies the WELLDOC property. Indeed, for any word $w\in {\mathcal A}^*$ and
any $m$ there exists a finite word $v$ such that if
$i_0, i_1, \dots, i_k$ denote the occurrences of $w$ in $v$, then
\[\left\{\left(|{\rm Pref}_{i_j} v|_0, \ldots, |{\rm Pref}_{i_j} v|_{d-1}\right)
\bmod m \mid j \in \{0,1,\dots,k\} \right\}=\zz_m^d\,.\]
Since $u$ is universal, $v$ is a factor of $u$. Denoting by $i$ an
occurrence of $v$ in $u$, one gets that the positions $i+i_j$ are
occurrences of $w$ in $u$. Hence
\begin{eqnarray*}&&\left\{\left(|{\rm Pref}_{i+i_j} u|_0, \ldots,
|{\rm Pref}_{i+i_j} u|_{d-1}\right)
\bmod m \mid j \in \{0,1,\dots,k\}\right\}=\\
&&=\left(|{\rm Pref}_{i} u|_0, \ldots, |{\rm Pref}_{i}
u|_{d-1}\right)+\\&&+\left\{\left(|{\rm Pref}_{i_j} v|_0, \ldots,
|{\rm Pref}_{i_j} v|_{d-1}\right) \bmod m \mid j \in \{0,1,\dots,k\}\right\}=\zz_m^d\,.\end{eqnarray*} \end{example}
Therefore, $u$ satisfies the WELLDOC property.

\subsection{Combination of PRNGs}
In order to eliminate the lattice structure, it helps to combine PRNGs in a~smart way. Such a~method was introduced in~\cite{GuPaPa1}.
Let $X=(X_n)_{n \in \mathbb N}$ and $Y=(Y_n)_{n \in \mathbb N}$ be two PRNGs with the same output $M \subset \mathbb N$ and the same period $m \in \mathbb N$, and let ${ u}=u_0u_1u_2\dots $ be a~binary infinite word over the alphabet
$\{0,1\}$.

\emph{The PRNG $Z=(Z_n)_{n \in \mathbb N}$
based on $u$} is obtained by the following algorithm:
\begin{enumerate}
\item Read step by step the letters of $u$.
\item When you read $0$ for the $i$-th time, copy the $i$-th symbol from $X$ to the end of the constructed sequence $Z$.
\item When you read $1$ for the $i$-th time, copy the $i$-th symbol from $Y$ to the end of the constructed sequence $Z$.
\end{enumerate}
This construction can be generalized for non-binary alphabets:
Using infinite words over a multiliteral alphabet, one can combine
more than two PRNGs. Remark that following terminology from
\cite{ss}, the sequence $Z$ is obtained as a \emph{shuffle} of the
sequences $X$ and $Y$ with the steering word $u$.

In order to distinguish between generators and infinite words used for their combination, we always denote generators with capital letters $X, Y, Z,\dots$ and words with lower-case letters $u, v, w$ (the same convention is applied for their outputs: $A,B,\dots$ for output values of generators (elements of $M$), $a,b,\dots$ for letters of words). Finite sequences of successive elements $\vv x=(X_i, X_{i+1},\dots, X_{i+t-1})$ of a~PRNG $X$ are called $t$-tuples, or vectors, while in the case of an infinite word $u$, we call $u_iu_{i+1}\dots u_{i+t-1}$ a~factor of length $t$.

\subsection{The WELLDOC Property and Absence of the Lattice Structure}\label{CCGAoLS}
Guimond et al. in~\cite{GuPaPa} have shown that PRNGs based on
infinite words coding a~certain class of cut-and-project sets have
no lattice structure. In the sequel, we will generalize their
result and find larger classes of words guaranteeing no lattice
structure for associated generators. We focus on the binary
alphabet, although everything works for multiliteral words as well
(and for combination of more generators therefore), since the
proofs become more technical in non-binary case.

%


\begin{theorem}\label{my}
Let $Z$ be the PRNG 
based on a~binary infinite word $u$ with the WELLDOC property. Then $Z$ has no lattice structure.
\end{theorem}
\begin{proof}
According to Proposition~\ref{za_zb}, it suffices to check that its assumptions are met.
%
Let $A,B \in M$ and $\ell \in \mathbb N$. Assume $A=X_i$ and
$B=Y_j$, where $X=(X_n)_{n \in \mathbb N}$ and $Y=(Y_n)_{n \in
\mathbb N}$ are the two combined PRNGs with the same output $M
\subset \mathbb N$ and the same period $m \in \mathbb N$. Consider
a~right special factor $w$ of $u$ of length $\ell$, i.e., both
words $w0$ and $w1$ are factors of $u$ (such a~factor $w$ exists
since $u$ is an aperiodic word because of the
WELLDOC property). By Definition~\ref{comb_cond}, it is possible to
find an occurrence $i_k$ of $w0$ in $u$ such that
$$|\Pre_{i_k}u|_0=i-|w|_0-1\bmod m, \qquad |\Pre_{i_k}u|_1=j-|w|_1-1\bmod m.$$
Reading the word $w0$ at the occurrence $i_k$, the corresponding $\ell$-tuple $(A_1,A_2,\dots, A_\ell)$ of the generator $Z$ consists of symbols
$$X_{(i-|w|_0)\bmod m}, \dots, X_{(i-1) \bmod m}\ \text{and} \ Y_{(j-|w|_1)\bmod m}, \dots, Y_{(j-1) \bmod m}.$$
When reading $0$ after $w$, the symbol $X_i=A$ from the first generator follows $(A_1,A_2,\dots, A_\ell)$.

Again, by Definition~\ref{comb_cond}, there exists an occurrence
$i_s$ of $w1$ in $u$ such that
$$|\Pre_{i_s}u|_0=i-|w|_0-1\bmod m, \qquad |\Pre_{i_s}u|_1=j-|w|_1-1\bmod m.$$
When reading the word $w$ at the occurrence $i_s$, the same $\ell$-tuple $(A_1,A_2,\dots, A_\ell)$ of $Z$ as previously occurs.
This time, however, $(A_1,A_2,\dots, A_\ell)$ is followed by $B$ because we read $w1$ and $Y_j=B$.
Thus, we have found an $\ell$-tuple $(A_1,A_2,\dots, A_\ell)$ of $Z$ followed in $Z$ by both $A$ and $B$.
\end{proof}
\begin{remark}\label{modifiedFib}
The WELLDOC property is sufficient, but not necessary for absence
of the lattice structure. For example, consider a~modified
Fibonacci word $\hat u$ where the letter $2$
is inserted after each letter, 
i.e., $\hat u=0212020212021202\dots$. It is easy to verify that
$\hat u$ does not have well distributed occurrences. 
However, we will show the following: Let $Z$ be the PRNG combining
three generators $X=(X_n)_{n \in \mathbb N}, Y=(Y_n)_{n \in
\mathbb N}$ and $V=(V_n)_{n \in \mathbb N}$ with the same output
$M \subset \mathbb N$ and the same period $m \in \mathbb N$
according to the modified Fibonacci word $\hat u$. Then $Z$ has no
lattice structure.

It suffices to verify assumptions of Proposition~\ref{za_zb}. Let
$A,B \in M$ and $\ell \in \mathbb N$, $\ell$ an even number (the
proof is analogous for odd $\ell$). Assume $A=X_i$ and $B=Y_j$.
Consider a~right special factor $w$ of the Fibonacci word $u$ of
length $\ell/2$. Since $u$ has the WELLDOC property, there exists
an occurrence $i_k$ of $w0$ in $u$ such that
$$|\Pre_{i_k}u|_0=i-|w|_0-1\bmod m, \qquad |\Pre_{i_k}u|_1=j-|w|_1-1\bmod m.$$
Then if we insert the letter $2$ after each letter of $w$, 
we obtain a~right special factor $\hat w$ of the modified Fibonacci word
$\hat u$ of length $\ell$. It holds then that
$$\begin{array}{rcl}|\Pre_{2i_k}\hat u|_0&=&i-|w|_0-1\bmod m=i-|\hat w|_0-1\bmod m,\\
|\Pre_{2i_k}\hat u|_1&=&j-|w|_1-1\bmod m=j-|\hat w|_1-1\bmod m,\\
|\Pre_{2i_k}\hat u|_2&=&i-|w|_0-1+j-|w|_1-1\bmod m=i+j-|\hat w|_2-2\bmod m.
\end{array}$$

When reading the word $\hat w0$ at the occurrence $2i_k$, the corresponding $\ell$-tuple $(A_1,A_2,\dots, A_\ell)$ of the generator $Z$ is followed by the symbol $X_i=A$ from the first generator.

Again, by the WELLDOC property of $u$, there exists an occurrence
$i_s$ of $w1$ in $u$ such that
$$|\Pre_{i_s}u|_0=i-|w|_0-1\bmod m, \qquad |\Pre_{i_s}u|_1=j-|w|_1-1\bmod m.$$
It holds then that
$$\begin{array}{rcl}|\Pre_{2i_s}\hat u|_0&=&i-|w|_0-1\bmod m=i-|\hat w|_0-1\bmod m,\\
|\Pre_{2i_s}\hat u|_1&=&j-|w|_1-1\bmod m=j-|\hat w|_1-1\bmod m,\\
|\Pre_{2i_s}\hat u|_2&=&i-|w|_0-1+j-|w|_1-1\bmod m=i+j-|\hat w|_2-2\bmod m.
\end{array}$$

When reading the word $\hat w$ at the occurrence $2i_s$, the same $\ell$-tuple $(A_1,A_2,\dots, A_\ell)$ of $Z$ as previously occurs.
This time, however, $(A_1,A_2,\dots, A_\ell)$ is followed by $B$ because we read $\hat w1$ and $Y_j=B$.
Thus, we have found an $\ell$-tuple $(A_1,A_2,\dots, A_\ell)$ of $Z$ followed in $Z$ by both $A$ and $B$.
Therefore $Z$ has no lattice structure.

\end{remark}

\begin{remark}
In the proof of Theorem~\ref{my}, the modulus $m$ from the WELLDOC
property is set to be equal to the period of the combined
generators. Therefore, if we require absence of the lattice
structure for a~PRNG obtained when combining PRNGs with a fixed
period $\hat m$, then it is sufficient to use an infinite word $u$
that satisfies the WELLDOC property for the modulus $m=\hat m$.
This means for instance that the Thue-Morse word is not completely
out of the game, but it cannot be used to combine periodic PRNGs
with the period being a~power of $2$.
\end{remark}

We have formulated a~combinatorial condition -- well distributed occurrences -- guaranteeing no lattice structure of the associated generator. It is now important to find classes of words satisfying such a condition.

\section{Sturmian Words}\label{PBoSWHNLS}

In this section we show that Sturmian words have well distributed occurrences.  

\begin{definition} An aperiodic infinite word $u$ is called \emph{Sturmian} if its factor complexity satisfies $\mathcal{C}_{u}(n)=n+1$ for all $n \in \mathbb N$. \end{definition} 

So, Sturmian words are by definition binary and they have the lowest possible factor complexity among aperiodic infinite words.
Sturmian words admit various types of characterizations of
geometric and combinatorial nature. One of such characterizations
is via irrational rotations on the unit circle. In \cite
{MorHed1940} Hedlund and Morse showed that each Sturmian word may
be realized measure-theoretically
by an irrational rotation on the
circle. That is, every Sturmian word is obtained by coding the
symbolic orbit of a point on the circle of circumference one
under a rotation $R_{\alpha}$ by an irrational angle%
\footnote{Measured by arc length (thus equivalent to $2\pi\alpha$ radians).}
 $\alpha$, $0<\alpha<1$, where the circle is
%
partitioned into two complementary intervals, one of length $\alpha $ and the other of
length $1-\alpha .$ Conversely, each such coding gives rise to
a Sturmian word.

\begin{definition} \label{rotation} The \emph{rotation} by angle $\alpha$ is the mapping
$R_{\alpha}$ from $[0,1)$ (identified with the unit circle) to
itself defined by $R_{\alpha}(x)=\{x+\alpha\}$, where
$\{x\}=x-\lfloor x\rfloor$ is the fractional part of $x$. Considering a
partition of $[0,1)$ into $I_0=[0, 1-\alpha)$, $I_1=[1- \alpha,
1)$, define a word
\[s_{\alpha, \rho}(n)=\begin{cases}0 & \mbox{ if }
R^n_{\alpha}(\rho)=\{\rho+n\alpha\} \in I_0, \\ 1 & \mbox{ if }
R^n_{\alpha}(\rho)=\{\rho+n\alpha\} \in I_1.
\end{cases}\]
One can also define $I'_0=(0, 1-\alpha]$, $I'_1=(1- \alpha, 1]$,
the corresponding word is denoted by $s'_{\alpha, \rho}$. \end{definition} 

%

Remark that some but not all Sturmian words are morphic.
In fact, it is known that a characteristic Sturmian word (i.e., $\rho=\alpha$) is morphic if and only if the continuous fraction expansion of $\alpha$ is periodic.
For more information on Sturmian words we refer to
\cite[Chapter~2]{Lo}.

\begin{theorem}\label{sturmian} Let $u$ be a Sturmian word.  Then $u$ has the WELLDOC property.
\end{theorem}

\begin{proof} In the proof we use the definition of Sturmian word via rotation.
The main idea is controlling the number of $1$'s modulo $m$ by taking circle of length $m$, and controlling the length taking the rotation by $m\alpha$.

For the proof we will use an equivalent reformulation of the
theorem:

\medskip

Let $u$ be a Sturmian word on $\{0,1\}$, for any natural number
$m$ and any factor $w$ of $u$ let us denote $i_0,i_1, \dots$ the
occurrences of $w$ in $u$. Then
\[\left\{ \left( i_j, |\Pre_{i_j}u|_1
\right) \bmod m \mid j \in \Nn \right\}=\mathbb Z_m^2.\]

\medskip

That is, we control the number of $1$'s and the length instead of
the number of $0$'s.

Since a Sturmian word can be defined via rotations by an irrational
angle on a unit circle, without loss of
generality we may assume that $u=s_{\alpha, \rho}$ for some
$0<\alpha<1$, $0\leq \rho <1$,
$\alpha$ irrational (see Definition \ref{rotation}).
Equivalently, we can consider $m$ copies of the circle connected
into one circle of length $m$ with $m$ intervals $I^i_1$ of length $\alpha$ corresponding to $1$. The Sturmian word is
obtained by rotation by $\alpha$ on this circle of length $m$ (see
Fig.~\ref{sturmian_fig}).

\begin{figure}[htb]
\begin{center}
\resizebox{6 cm}{!}{\includegraphics{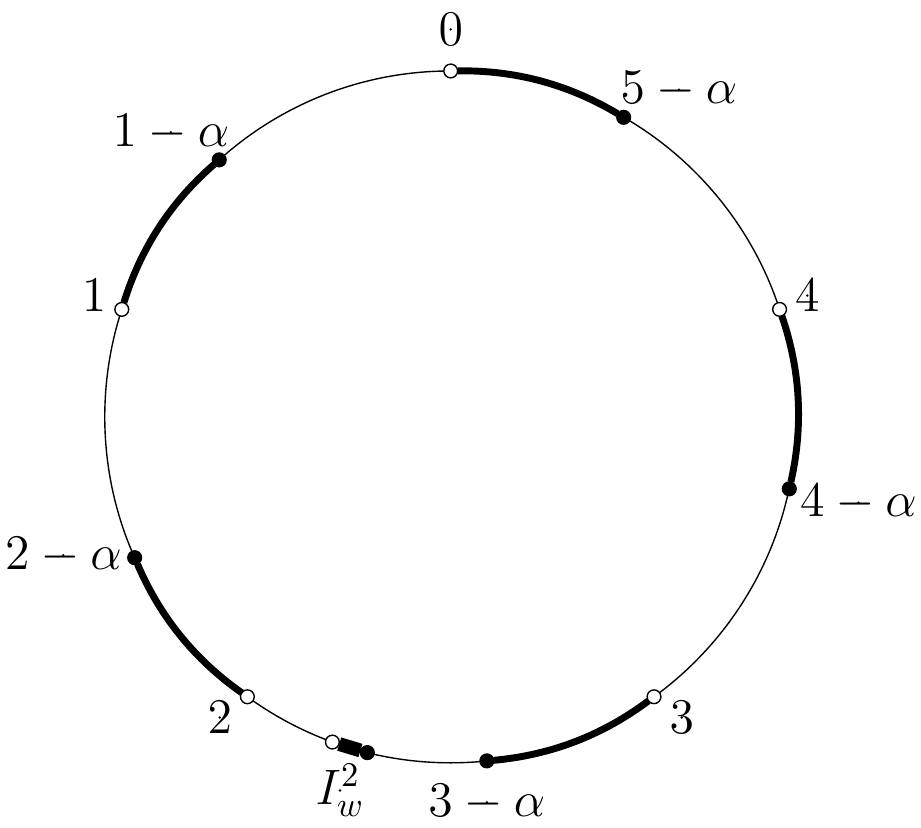}}
\caption{Illustration to the proof of Theorem \ref{sturmian}: the
example for $m=5$. }\label{sturmian_fig}
\end{center}
\end{figure}

Namely, we define the rotation $R_{\alpha,m}$ as the mapping
 from $[0,m)$ (identified with the circle of length $m$) to
itself defined by $R_{\alpha,m}(x)=\{x+\alpha \}_m$, where
$\{x\}_m=x-\lfloor x/m\rfloor m$ and for $m=1$ coincides with the fractional
part of $x$. A~partition of $[0,m)$ into $2m$ intervals $I^i_0=[i,
i+1-\alpha)$, $I^i_1=[i+1- \alpha, i+1)$, $i=0, \dots, m-1$
defines the Sturmian word $u=s_{\alpha, \rho}$:
\[s_{\alpha, \rho}(n)=\begin{cases}0 & \mbox{ if }
R^n_{\alpha,m}(\rho)=\{\rho+n\alpha\}_m \in I^i_0 \mbox{ for some } \hskip2pt i=0,\dots, m-1, \\
1 & \mbox{ if } R^n_{\alpha,m}(\rho)=\{\rho+n\alpha\}_m \in I^i_1
\mbox{ for some } \hskip2pt i=0,\dots, m-1.
\end{cases}\]

It is well known that any factor $w=w_{0}\cdots w_{k-1}$ of $u$ corresponds to an
interval $I_w$ in $[0,1)$%
, so that whenever you
start rotating from the interval $I_w$, you obtain $w$. Namely,
$x\in I_w$ if and only if $x \in I_{w_0},  R_{\alpha}(x) \in
I_{w_1},  \dots, R^{|w|-1}_{\alpha}(x) \in I_{w_{|w|-1}}$.

Similarly, we can define $m$ intervals corresponding to $w$ in
$[0,m)$ (circle of length $m$), so that if $I_w=[x_1,x_2)$, then
$I^i_w=[x_1+i,x_2+i)$, $i=0, \dots, m-1$.

\smallskip

Fix a factor $w$ of $u$, take arbitrary $(j, i) \in
\mathbb Z_m^2$. Now let us organize $(j,i)$ among the
occurrences of $w$, i.e., find $l$ such that $u_l \dots
u_{l+|w|-1} = w$, $l \mod m = j$ and $|{\rm {Pref}}_l u |_1 \mod m
= i$:

Consider rotation $R_{m\alpha,m}(x)$ by $m\alpha$ instead of
rotation by $\alpha$, and start $m$-rotating from $j\alpha+\rho$. 
Formally, $R_{m\alpha,m}(x)=\{x+m\alpha \}_m$, where, as
above, $\{x\}_m=x-[x/m]m$. This rotation will put us to positions
$mk+j$, $k \in \mathbb{N}$, in the Sturmian word: for $a\in\{0,1\}$
one has
$s_{\alpha, \rho}(mk+j)=a$ if
$R^k_{m\alpha,m}(j\alpha+\rho)=\{j\alpha+\rho+km\alpha\}_m \in I^i_a$
for some $i=0,\dots, m-1$.

Remark that the points in the orbit of an $m$-rotation of a point
on the $m$-circle are dense, and hence the rotation comes
infinitely often to each interval. So pick $k$ when
$j\alpha+mk\alpha+\rho \in I^i_w\subset [i,i+1)$ (and actually
there exist infinitely many such $k$).
 Then the length $l$ of the corresponding prefix is equal to $km+ j$, and the number of $1$'s in it is $i+mp$, where $p$ is the number of complete circles you made, i.e.,
$p=[( j\alpha+mk\alpha+\rho)/m]$.

\end{proof}

\section{Arnoux-Rauzy Words}
In this section we show that Arnoux-Rauzy words~\cite{ArRa}, which are natural extensions of Sturmian words to larger alphabets, also satisfy the WELLDOC property. Note that the proof for Sturmian words cannot be generalized to Arnoux-Rauzy words, because it is based on the geometric interpretation of Sturmian words via rotations, while this interpretation does not extend to Arnoux-Rauzy words.

\subsection{Basic Definitions}
The definitions and results we remind in this subsection are well-known and mostly taken from~\cite{ArRa,DJP} and generalize the ones given for binary words in~\cite{deLuca}.

\begin{definition}
Let ${\mathcal A}$ be a finite alphabet. The \emph{reversal operator} is the operator $\sim:{\mathcal A}^* \mapsto {\mathcal A}^*$ defined by recurrence in the following way:
\[ \tilde \varepsilon = \varepsilon, \quad \widetilde{va}=a\tilde v\]
for all $v \in {\mathcal A}^*$ and $a \in {\mathcal A}$. The fixed points of the reversal operator are called \emph{palindromes}. \end{definition}

\begin{definition}
Let $v \in {\mathcal A}^*$ be a finite word over the alphabet ${\mathcal A}$. The \emph{right palindromic closure} of $v$, denoted by $v^{(+)}$, is the shortest palindrome that has $v$ as a prefix. It is readily verified that if $p$ is the longest palindromic suffix of $v=wp$, then $v^{(+)}=wp\tilde{w}$.
\end{definition}

\begin{definition}
We call the \emph{iterated (right) palindromic closure operator} the operator $\psi$ recurrently defined by the following rules:
\[ \psi(\varepsilon)=\varepsilon, \quad \psi(va)=(\psi(v)a)^{(+)}\]
for all $v \in {\mathcal A}^*$ and $a \in {\mathcal A}$.
The definition of $\psi$ may be extended to infinite words $u$ over ${\mathcal A}$ as $\psi(u)=\lim_{n} \psi(\Pre_n u)$,
i.e., $\psi(u)$ is the infinite word having $\psi(\Pre_n u)$ as its prefix for every $n \in \mathbb N$.
\end{definition}

\begin{definition}\label{defAR}
Let $\Delta$ be an infinite word on the alphabet ${\mathcal A}$
such that every letter occurs infinitely often in $\Delta$. The
word $c=\psi(\Delta)$ is then called a \emph{characteristic (or
standard) Arnoux-Rauzy word} and $\Delta$ is called the
\emph{directive sequence} of $c$. An infinite word $u$ is called
an Arnoux-Rauzy word if it has the same set of factors as a
(unique) characteristic Arnoux-Rauzy word, which is called the
characteristic word of $u$. The directive sequence of an
Arnoux-Rauzy word is the directive sequence of its characteristic
word.
\end{definition}

Let us also recall the following well-known characterization (see e.g.~\cite{DJP}): 

\begin{theorem}\label{charAR}
Let $u$ be an aperiodic infinite word over the alphabet ${\mathcal A}$. Then $u$ is a standard Arnoux-Rauzy word if and only if the following hold:
\begin{enumerate}
\item $\Ff(u)$ is closed under reversal (that is, if $v$ is a factor of $u$ so is $\tilde v$).
\item Every left special factor of $u$ is also a prefix.
\item If $v$ is a right special factor of $u$ then $va$ is a factor of $u$ for every $a \in {\mathcal A}$.
\end{enumerate}
\end{theorem}

From the preceding theorem, it can be easily verified that the bispecial factors of a standard Arnoux-Rauzy correspond to its palindromic prefixes (including the empty word), and hence to the iterated palindromic closure of the prefixes of its directive sequence. That is, if
\[ \varepsilon=b_0,b_1,b_2,\ldots \]
is the sequence, ordered by length, of bispecial factors of the standard Arnoux-Rauzy word $u$, $\Delta=\Delta_0\Delta_1\cdots$ its directive sequence (with $\Delta_i \in {\mathcal A}$ for every $i$), we have $b_{i+1}=(b_i\Delta_i)^{(+)}$.

A direct consequence of this, together with the preceding definitions, is the following statement, which will be used in the sequel.
\begin{lemma} \label{lemmabispecials}
Let $u$ be a characteristic Arnoux-Rauzy word and let $\Delta$ and $(b_i)_{i \geq 0}$ be defined as above. If $\Delta_i$ does not occur in $b_i$, then $b_{i+1}=b_i\Delta_i b_i$. Otherwise let $j<i$ be the largest integer such that $\Delta_j=\Delta_i$. Then $b_{i+1}=b_i b_j^{-1} b_i$.
\end{lemma}

\subsection{Parikh Vectors and Arnoux-Rauzy Factors}

Where no confusion arises, given an Arnoux-Rauzy word $u$, we will denote by
\[ \varepsilon=b_0, b_1, \ldots, b_n, \ldots \]
the sequence of bispecial factors of $u$ ordered by length and we will denote for any $i \in \Nn$, $\vv b_i$ the Parikh vector of $b_i$.

\begin{remark}
By the pigeonhole principle, it is clear that for every $m \in \Nn$ there exists an integer $N \in \Nn$ such that, for every $i\geq N$, the set $\{ j>i \mid \vv b_j \equiv_m \vv b_i \}$ is infinite. Where no confusion arises and with a slight abuse of notation, fixed $m$, we will always denote by $N$ the smallest of such integers.
\end{remark}

\begin{lemma}\label{zerocombination}
Let $u$ be a characteristic Arnoux-Rauzy word and let $m \in \Nn$. Let
\[ \alpha_1 \vv b_{j_1} + \cdots + \alpha_k \vv b_{j_k} \equiv_m   \vv v \in \zz_m^d\]
be a linear combination of Parikh vectors such that $\sum_{i=1}^k \alpha_i = 0$, with $j_i \geq N$ and $\alpha_i \in \zz$ for all  $i \in \{1, \ldots k\}$.
Then, for any $\ell \in \Nn$, there exists a prefix $v$ of $u$ such that the Parikh vector of $v$ is congruent to $\vv v$ modulo $m$ and $vb_\ell$ is also a prefix of $u$.
\end{lemma}
\begin{proof}
Without loss of generality, we can assume $\alpha_1 \geq \alpha_2 \geq \cdots \geq \alpha_k$, hence there exists $k'$ such that  \[ \alpha_1 \geq \alpha_{k'} \geq 0 \geq \alpha_{k'+1} \geq \alpha_k. \] We will prove the result by induction on $\beta = \sum_{j=1}^{k'} \alpha_j$. If $\beta = 0$, trivially, we can take $v = \varepsilon$ and the statement is clearly verified. Let us assume the statement true for all $0\leq \beta < n$ and let us prove it for $\beta = n$. By the remark preceding this lemma, for every $\ell$ we can choose $i' > j' > \ell$ such that $\vv b_{j_1} \equiv_m \vv b_{i'}$ and $\vv b_{j_k} \equiv_m \vv b_{j'}$. Since every bispecial factor is a prefix and suffix of all the bigger ones, in particular we have that
$b_{j'}$ is a suffix of $b_{i'}$, and $b_{\ell}$ is a prefix of $b_{j'}$;
this implies that $b_{i'}b_{j'}^{-1}b_\ell$ is actually a prefix of $b_{i'}$. By assumption, the Parikh vector of $b_{i'}b_{j'}^{-1}$ is clearly $\vv b_{i'}-\vv b_{j'}\equiv_m \vv b_{j_1} - \vv b_{j_k}$. Since $\alpha_1 \geq 1$ implies $\alpha_k \leq -1$, we have, by induction hypothesis, that there exists a prefix $w$ of $u$ such that the Parikh vector of $w$ is congruent modulo $m$ to
\[ (\alpha_1 -1) \vv b_{j_1} + \cdots + (\alpha_k +1) \vv b_{j_k} \]
and $wb_{i'}$ is a prefix of $u$. Hence $wb_{i'}b_{j'}^{-1}b_\ell$ is also a prefix of $u$ and, by simple computation, the Parikh vector of $v=wb_{i'}b_{j'}^{-1}$ is congruent modulo $m$ to $\vv v=\alpha_1 \vv b_{j_1} + \cdots + \alpha_k \vv b_{j_k}$.
\end{proof}

\begin{definition}
Let $n \in \zz$. We will say that an integer linear combination of integer vectors is a \emph{$n$-combination} if the sum of all the coefficients equals $n$.
\end{definition}

\begin{lemma}\label{keylemma}
Let $u$ be a characteristic Arnoux-Rauzy word and let $n \in \Nn$. Every $n$-combination of Parikh vectors of bispecial factors can be expressed as an $n$-combination of Parikh vectors of arbitrarily large bispecials. In particular, for every $K, L \in \Nn$, it is possible to find a finite number of integers $\alpha_1, \ldots, \alpha_k$ such that $\vv b_K = \alpha_1 \vv b_{j_1} + \cdots + \alpha_k \vv b_{j_k}$ with $j_i > L$ for every $i$ and $\alpha_1+\cdots+\alpha_k=1$.
\end{lemma}

\begin{proof}
A direct consequence of Lemma \ref{lemmabispecials} is that for every $i$ such that $\Delta_i$ appears in $b_i$, we have $\vv b_{i+1}=2\vv b_i-\vv b_j$, where $j<i$ is the largest such that $\Delta_j=\Delta_i$. This in turn (since every letter in $\Delta$ appears infinitely many times from the definition of Arnoux-Rauzy word) implies that \emph{for every} non-negative integer $j$, there exists a positive $k$ such that $\vv b_j = 2\vv b_{j+k}-\vv b_{j+k+1}$, that is, we can substitute each Parikh vector of a bispecial with a $1$-combination of Parikh vectors of strictly larger bispecials. Simply iterating the process, we obtain the statement.
\end{proof}

In the following we will assume the set ${\mathcal A}$ to be a finite
alphabet of cardinality $d$. For every set $X \subseteq {\mathcal A}^*$ of
finite words, we will denote by $\pv X \subseteq \zz^d$ the set of
Parikh vectors of elements of $X$ and for every $m \in \Nn$ we
will denote by $\pvn X \subseteq \zz_m^d$  the set of elements of
$\pv X$ reduced modulo $m$.

For an infinite word $u$ over ${\mathcal A}$, and a factor $v$ of $ u$, let $S_v( u)$ denote the set of all prefixes of $ u$ followed by an occurrence of $v$. In other words,
\[S_v( u)=\{p \in \pref{ u} \mid p v \in \pref{ u} \}.\]

\begin{definition}
For any set of finite words $X \subseteq {\mathcal A}^*$, we will say that $ u$ \emph{has the property} $\Pp_X$ (or, for short, that $ u$ has $\Pp_X$) if, for every $m \in \Nn$ and for every $v \in X$ we have that
\[ \pvn{S_v( u)}= \zz_m^d.\]
That is to say, for every vector $\vv w \in \zz_m^d$ there exists a word $w \in S_v( u)$ such that the Parikh vector of $w$ is congruent to $\vv w$ modulo $m$.
\end{definition} 
With this notation, an infinite word $u$ has the WELLDOC property if and only if it has the
property $\Pp_{\Ff(u)}$.

\begin{proposition}\label{prefresult}
Let $u$ be a characteristic Arnoux-Rauzy word over the $d$-letter alphabet ${\mathcal A}$. Then $u$ has the property $\Pp_{\pref u}$.
\end{proposition}
\begin{proof}
Let us fix an arbitrary $m \in \Nn$. We want to show that, for every $v \in \pref u$, $\pvn{S_v(u)}=\zz_m^d$. Let then $\vv v \in \zz^d$ and $\ell$ be the smallest number such that $v$ is a prefix of $b_\ell$. Let $i_1<i_2< \cdots < i_d$ be such that $\Delta_{i_j}$ does not appear in $b_{i_j}$, where $\Delta$ is the directive word of $u$. Without loss of generality, we can rearrange the letters so that each $\Delta_{i_j}$ is lexicographically smaller than $\Delta_{i_{j+1}}$. With this assumption if, for every $j$, we set $\vv v_j=\vv b_{i_j +1}$, i.e., equal to the Parikh vector of $b_{i_j +1}$, which, by the first part of Lemma \ref{lemmabispecials}, equals $b_{i_j}\Delta_{i_j}b_{i_j}$, we can find $j-1$ positive integers $\mu_1, \ldots, \mu_{j-1}$ such that $\vv v_j = (\mu_1,\mu_2,\ldots,\mu_{j-1},1,0,\ldots,0)$. It is easy to show, then, that the set $V=\{\vv v_1,\ldots, \vv v_d\}$ generates $\zz^d$, hence there exists an integer $n$ such that $\vv v$ can be expressed as an $n$-combination of elements of $V$ (which are Parikh vectors of bispecial factors of $u$). Trivially, then, $\vv v = \vv v- n \vv 0 = \vv v -n \vv b_0$; thus, it is possible to express $\vv v$ as a $0$-combination of Parikh vectors of (by the previous Lemma \ref{keylemma}) arbitrarily large bispecial factors of $u$. By Lemma \ref{zerocombination}, then there exists a prefix $p$ of $u$ whose Parikh vector $\vv p$ satisfies $\vv p \equiv_m \vv v$ and $pb_\ell$ is a prefix of $u$. Since we picked $\ell$ such that $v$ is a prefix of $b_\ell$, we have that $p \in S_v(u)$. From the arbitrariness of $v$, $\vv v$ and $m$, we obtain the statement.
\end{proof}

As a corollary of Proposition~\ref{prefresult}, we obtain the main result of this section.
\begin{theorem}\label{AR}
Let $u$ be an Arnoux-Rauzy word over the $d$-letter alphabet ${\mathcal A}$. Then $u$ has the property $\Pp_{\Ff(u)}$, or equivalently, $u$ has the WELLDOC property.
\end{theorem}
\begin{proof}
Let $m$ be a positive integer and let $c$ be the characteristic word of $u$. Let $v$ be a factor of $u$ and $x v y$ be the shortest bispecial containing $v$. By Proposition \ref{prefresult}, we have that $\pvn {S_{x v}(c)}= \zz_m^d$ and, since the set is finite, we can find a prefix $p$ of $c$ such that $\pvn {S_{x v}(p)}= \zz_m^d$. Let $w$ be a prefix of $u$ such that $wp$ is a prefix of $u$. If $\vv x$ and $\vv w$ are the Parikh vectors of, respectively, $x$ and $w$, it is easy to see that
\[ \vv w + \vv x + \pv{S_{xv}(p)}\subseteq  \vv w + \pv{S_v(p)} \subseteq \pv{S_v(u)} \]
Since we have chosen $p$ such that $\pvn {S_{x v}(p)}= \zz_m^d$, we clearly obtain that $\pvn{S_v(u)}=\zz_m^d$ and hence, by the arbitrariness of $v$ and $m$, the statement.
\end{proof}
%

\begin{remark} Now we introduce a~simple method of obtaining words satisfying the WELLDOC property. Take a word $u$ with the WELLDOC property over an alphabet $\{0,1,\dots,d-1\}$, $d > 2$, apply a morphism $\varphi: d-1\mapsto 0, i\mapsto i $ for $i=0,\dots, d-2$, i.e., $\varphi$ joins two letters into one. It is straightforward that $\varphi(u)$ has the WELLDOC property. So, taking Arnoux-Rauzy words and joining some letters, we obtain other words than Sturmian and Arnoux-Rauzy satisfying the WELLDOC property. \end{remark} 

\begin{remark}
Now we introduce another class of morphisms preserving the WELLDOC
property. Recall that the \emph{adjacency matrix} $\Phi$ of a
morphism $\varphi:\Aa\to\Aa$, with $\Aa=\{0,1,\ldots,d-1\}$, is
defined by $\Phi_{i,j}=|\varphi(j-1)|_{i-1}$ for $1\leq i,j\leq
d$. By definition, it follows that if $\vv v$ is the Parikh vector
of $v\in\Aa^{*}$, then $\Phi\vv v$ is the Parikh vector of
$\varphi(v)$.

Let us show that if $\det\Phi=\pm 1$ and $u$ has the WELLDOC
property, then so does $\varphi(u)$. Indeed, let $w$ be any factor of $\varphi(u)$, and
suppose $xwy=\varphi(v)$ for some $v\in\Ff(u)$ and
$x,y\in\Aa^{*}$. We then have
$S_{w}(\varphi(u))\supseteq \varphi(S_{v}(u))x$,
so that,
writing $\vv x$ for the Parikh vector of $x$, we have for any $m>0$
\[\pvn{S_{w}(\varphi(u))}\supseteq\Phi\cdot \pvn{S_{v}(u)}+\vv{x}\bmod m\,.\]
Since $u$ has the WELLDOC property, $\pvn{S_{v}(u)}=\zz_{m}^{d}$. As $\det\Phi=\pm 1$,
$\Phi$ is invertible (even modulo $m$), so that
$\Phi\cdot\zz_{m}^{d}+\vv x\bmod m=\zz_{m}^{d}$.
Hence $\pvn{S_{w}(\varphi(u))}=\zz_{m}^{d}$, showing that $\varphi(u)$ has the WELLDOC
property by the arbitrariness of $w$ and $m$.
\end{remark} 

\section{Statistical Tests of PRNGs}
In the previous part, we have explained that PRNGs based on
infinite words with well distributed occurrences have no lattice
structure. In this sequel we demonstrate this by empirical statistical
tests. We have chosen to use LCGs as underlying generators \emph{explicitly} for their known weaknesses. We will show how mixing based on aperiodic infinite words will cope with these weaknesses and whether statistical tests will show any significant improvements.

\subsection{Computer Generation of Morphic Words}
Any real computer is a~finite state machine and hence it can
generate only finite prefixes of infinite words. From practical
point of view it is important to find algorithms that are
efficient both in memory footprint and CPU time. In~\cite{Patera}
an efficient algorithm for generating the Fibonacci word was
introduced: The prefix of length $n$ is generated in $O(\log(n))$
space and $O(n)$ time. We generalize this method for any Sturmian
and Arnoux-Rauzy word being a~fixed point of a~morphism $\varphi$.
The main ingredient is that we consider $\varphi^n$ instead of $\varphi$; we precompute and store in the memory $\varphi^n(a)$ for any $a \in {\mathcal A}$.
The runtime to generate $10^{10}$ letters of the Fibonacci and the
Tribonacci word is summarized in Table~\ref{table:speed}.
There are the following observations we would like to point out:
\begin{enumerate}
\item There is no need to store the first $n$ letters in memory to generate the $(n+1)$-th letter. Letters are generated on the fly and only nodes of the traversal tree are kept in the memory. Memory consumption needed to generate the first $10^{10}$ letters is shown in Table~\ref{table:speed}. The algorithm also supports leap frogging, generation can be started at any position in the word. The consequence is that the algorithm can be easily parallelized to produce multiple streams~\cite{Ecuyer}.
\item Using the method from~\cite{Patera} together with our improvement for generation of Sturmian and Arnoux-Rauzy words, the speed of generation of their prefixes is much higher than the speed of generation of LCGs output values. For example, generation of $10^{10}$ $32$-bit values using a~LCG modulo $2^{64}$ takes $14.3$ seconds on our machine. Compare it to $0.5$ seconds for generation of $10^{10}$ letters of a~fixed point of a~morphism with the same hardware. Thus, using a~fixed point to combine LCGs causes only a~negligible runtime penalty.
\item The speed of generation can be further improved by using a~higher initial memory footprint and CPU that can effectively copy such larger chunks of memory (size of L1 data cache is a limiting factor). Thus the new method scales nicely and can benefit form the future CPUs with higher L1 caches. The only requirement is to precompute $\varphi^n(a)$, $a \in {\mathcal A}$, for larger $n$. Our program does this automatically based on the limit on the initial memory consumption provided by the user.
\end{enumerate}

\begin{table}[ht]
\centering
\setlength{\tabcolsep}{5pt}
\begin{tabular}{|c|c|c|}
\hline
Word & Fibonacci & Tribonacci \\
\hline
$\varphi$ morphism rule & 115s / 336 Bytes & 107s / 256 Bytes \\
\hline
$\varphi^n$ morphism rule & 0.41s / 32 Bytes & 0.36s / 32 Bytes  \\
\hline
\end{tabular}
\caption{The comparison of time in seconds and memory consumption to hold the traversal tree state needed to generate the first $10^{10}$ letters of the Fibonacci and the Tribonacci word using the original~\cite{Patera} ($1$st line) and the new algorithm ($2$nd line). The iteration $n$~in~the $\varphi^n$ rule was chosen so that the length of $\varphi^n(a)$ does not exceed $4096$ bytes for any $a \in {\mathcal A}$. The measurement was done on Intel Core i7-3520M CPU running at 2.90GHz.}
\label{table:speed}
\end{table}
\subsection{Testing PRNGs Based on Sturmian and Arnoux-Rauzy words}
We will present results for PRNGs based on:
\begin{itemize}
\item the Fibonacci word (as an example of a Sturmian word), i.e., the fixed point of the morphism $0 \mapsto 01, 1 \mapsto 0$,
\item the modified Fibonacci word -- Fibonacci2 -- with the letter $2$ inserted after each letter (see Remark~\ref{modifiedFib}),
\item the Tribonacci word (as the simplest example of a ternary Arnoux-Rauzy word), i.e., the fixed point of $0 \mapsto 01, 1 \mapsto 02, 2 \mapsto 0$.
\end{itemize}
We have implemented PRNGs for more morphic Sturmian words and
ternary Arnoux-Rauzy words. Since the results are similar, we
present in the sequel only the above three representatives. Our
program generating PRNGs based on morphic words is available
online, together with a~description \cite{APRNG-lib}.

Remark that we included the modified Fibonacci word that does not have the
WELLDOC property, but at the same time it guarantees no lattice
structure for the arising generator. The reason for including
it is that we would like to illustrate that such a~word
leads to worse results in testing than the Fibonacci word.


\subsubsection{Combining LCGs}
Instead of combining plain LCGs, we will execute some modifications before their combination. Those modifications turn out to be useful according to the known weaknesses of LCGs. 

We have chosen LCGs with the period $m$ in range from $2^{47}-115$ to $2^{64}$, but we use only their upper $32$ bits as the output since the statistical tests require $32$-bit sequences as the input. Their output is thus in all cases $M=\{0,1,\dots, 2^{32}-1\}$.

We use two batteries of random tests -- TestU01 BigCrush and
PractRand. They operate differently. The first one includes $160$
statistical tests, many of them tailored to the specific classes
of PRNGs. It is a~reputable test, however its drawback is that it
works with a fixed amount of data and discards the least
significant bit (for some tests even two bits) of the $32$-bit
numbers being tested. The second battery consists of three
different tests where one is adapted on short range correlations,
one reveals long range violations, and the last one is a~variation
on the classical Gap test. Details can be found
in~\cite{PractRand-engines,PractRand-discussion}. Moreover, the
PractRand battery applies automatically various filters on the
input data. For our purpose the lowbit filter is interesting -- it
is passing various number of the least significant bits to the
statistical tests. As we have already mentioned, the LCGs with
$m=2^{\ell}$ have a~much shorter period than the LCG itself.
Therefore the lowbit filter is useful to check whether this
weakness disappears when LCGs are combined according to an
infinite word. The PractRand tests are able to treat very long
input sequences, up to a few exabytes. To control the runtime we
have limited the length of input sequences to $16$TB.

The first column of Table~\ref{table:LCG} shows the list of tested
LCGs. The BigCrush column shows how many tests of the TestU01
BigCrush battery failed. The PractRand column gives the $\log_{2}$
of sample datasize in Bytes for which the results of the PractRand
tests started to be ``very suspicious'' ($p$-values smaller than
$10^{-5}$). One LCG did not show any failures in the PractRand
tests which is denoted as $>44$ -- the meaning is that the
PractRand test has passed successfully $16$TB of input data and
the test was stopped there. The last column provides time in
seconds to generate the first $10^{10}$ $32$-bit sequences of
output on Intel i7-3520M CPU running at 2.90GHz.

\begin{table}[ht]
\centering
\setlength{\tabcolsep}{5pt}
\renewcommand{\arraystretch}{1.3}
\resizebox{\textwidth}{!} {
\begin{tabular}{|l|l|c|c|c|}
\hline
Generator & Legend & BigCrush & PractRand & Time $10^{10}$\\
\hline
LCG($2^{47}-115,71971110957370,0)$ & L47-115 & 14 & 40 & 281\\
\hline
LCG($2^{63}-25,2307085864,0)$ & L63-25 & 2 & $>$44 & 277\\
\hline
LCG($2^{59},13^{13},0)$ & L59 & 19 & 27 & 14.1\\
\hline
LCG($2^{63},5^{19},1)$ & L63  & 19 & 33 & 14.4\\
\hline
LCG($2^{64},2862933555777941757,1)$ & L64\_28 & 18 & 35 & 14.0\\
\hline
LCG($2^{64},3202034522624059733,1)$ & L64\_32 & 14 & 34 & 14.1\\
\hline
LCG($2^{64},3935559000370003845,1)$ & L64\_39 & 13 & 33 & 14.0\\
\hline
\end{tabular}
}
\caption{List of the used LCGs with parameters LCG$(m,a,c)$. Results in the BigCrush (number of failed tests) and in the PractRand ($\log_{2}$ of sample size for which the test started to fail) battery of statistical tests. Time in seconds to generate the first $10^{10}$ $32$-bit words of output on Intel i7-3520M CPU running at 2.90GHz. }
\label{table:LCG}
\end{table}
From Table~\ref{table:LCG} it can be seen that the LCGs with $m\in \{2^{47}-115, 2^{63}-25\}$ have the best statistical properties from the chosen LCGs. At the same time, these LCGs are $20$ times slower than the other LCGs used. This is because we have used $128$-bit integer arithmetic to compute their internal state and because explicit modulo operation cannot be avoided.  As the CPU used does not have the $128$-bit integer arithmetic, it has to be implemented in software (in this case via GCC's {\tt \_\_int128} type) which is much slower than the $64$-bit arithmetic wired on CPU.


\subsubsection{Results in Statistical Tests}
We will present results for the PRNGs based on the Fibonacci, Fibonacci2 and Tribonacci word using the different combinations of LCGs from Table~\ref{table:LCG}. It includes also the situations where the instances of the same LCG are used. Each instance has its own state. The LCGs were seeded with the value $1$. The PRNGs were warmed up by generating $10^{9}$ values before statistical tests started. Since the relative frequency of the letters in the aperiodic words differ a lot (for example for the Fibonacci word the ratio of zeroes to ones is given by $\tau=\frac{1+\sqrt{5}}{2}$), 
the warming procedure will guarantee that the state of instances of LCGs will differ even when the same LCGs are used. Even more importantly, the distance between the LCGs is growing as the new output of PRNGs is generated.

Summary of results is in Table~\ref{table:APRNG_Table}. The BigCrush column is using the following notation: the first number indicates how many tests from the BigCrush battery have clearly failed and the optional second number in parenthesis denotes how many tests have suspiciously low $p$-value in the range from $10^{-6}$ to $10^{-4}$. The PractRand column gives the $\log_{2}$ of sample datasize in Bytes for which the results of the PractRand tests started to be ``very suspicious'' ($p$-values smaller than $10^{-5}$). The maximum sample data size used was $16$TB $\doteq 2^{44}$B. The Time column gives runtime in seconds to generate the first $10^{10}$ $32$-bit words of output on Intel i7-3520M CPU running at 2.90GHz. The source code of the testing programs is in~\cite{APRNG-lib}.

\LTcapwidth=\textwidth
\begin{longtable}{|l|l|l|l|l|c|c|c|}

\hline \multicolumn{1}{|c|}{Word} & \multicolumn{1}{c|}{Group} & \multicolumn{1}{c|}{0} & \multicolumn{1}{c|}{1} & \multicolumn{1}{c|}{2} & BigCrush & PractRand & Time $10^{10}$ \\ \hline
\endfirsthead

\multicolumn{8}{l}%
{{\tablename\ \thetable{} -- Continued from the previous page}} \\
 \hline \multicolumn{1}{|c|}{Word} & \multicolumn{1}{c|}{Group} & \multicolumn{1}{c|}{0} & \multicolumn{1}{c|}{1} & \multicolumn{1}{c|}{2} & BigCrush & PractRand & Time $10^{10}$ \\ \hline
\endhead

\hline \multicolumn{8}{|r|}{{Continued on the next page}} \\ \hline
\endfoot

\caption{Summary of results of statistical tests for PRNGs based on the Fibonacci, Fibonacci2 and Tribonacci word and different combinations of LCGs from Table~\ref{table:LCG}.}
\label{table:APRNG_Table}
\endlastfoot

\multicolumn{ 1}{|c|}{Fib} & \multicolumn{ 1}{c|}{A} & L64\_28 & L64\_28 &  & 0 & 41 & 30.2 \\ \cline{ 3- 8}
\multicolumn{ 1}{|l|}{} & \multicolumn{ 1}{l|}{} & L64\_32 & L64\_28 &  & 0(1) & 41 & 29.3 \\ \cline{ 3- 8}
\multicolumn{ 1}{|l|}{} & \multicolumn{ 1}{l|}{} & L64\_39 & L64\_28 &  & 0 (2) & 41 & 31 \\ \cline{ 3- 8}
\multicolumn{ 1}{|l|}{} & \multicolumn{ 1}{l|}{} & L64\_28 & L64\_32 &  & 0 & 41 & 30.2 \\ \cline{ 3- 8}
\multicolumn{ 1}{|l|}{} & \multicolumn{ 1}{l|}{} & L64\_32 & L64\_32 &  & 0 & 41 & 30.1 \\ \cline{ 3- 8}
\multicolumn{ 1}{|l|}{} & \multicolumn{ 1}{l|}{} & L64\_39 & L64\_32 &  & 0 & 41 & 30.1 \\ \cline{ 3- 8}
\multicolumn{ 1}{|l|}{} & \multicolumn{ 1}{l|}{} & L64\_28 & L64\_39 &  & 0 & 42 & 30.2 \\ \cline{ 3- 8}
\multicolumn{ 1}{|l|}{} & \multicolumn{ 1}{l|}{} & L64\_32 & L64\_39 &  & 0 & 40 & 30.5 \\ \cline{ 3- 8}
\multicolumn{ 1}{|l|}{} & \multicolumn{ 1}{l|}{} & L64\_39 & L64\_39 &  & 0 & 42 & 30.1 \\ \cline{ 2- 8}
\multicolumn{ 1}{|l|}{} & \multicolumn{ 1}{c|}{B} & L47-115 & L47-115 &  & 1(1) & $>$44 & 302 \\ \cline{ 3- 8}
\multicolumn{ 1}{|l|}{} & \multicolumn{ 1}{l|}{} & L63-25 & L63-25 &  & 0(1) & $>$44 & 299 \\ \cline{ 3- 8}
\multicolumn{ 1}{|l|}{} & \multicolumn{ 1}{l|}{} & L59 & L59 &  & 0(1) & 34 & 28.7 \\ \cline{ 3- 8}
\multicolumn{ 1}{|l|}{} & \multicolumn{ 1}{l|}{} & L63 & L63 &  & 0 & 40 & 29.8 \\ \cline{ 2- 8}
\multicolumn{ 1}{|l|}{} & \multicolumn{ 1}{c|}{C} & L63-25 & L59 &  & 0 & 38 & 198 \\ \cline{ 3- 8}
\multicolumn{ 1}{|l|}{} & \multicolumn{ 1}{l|}{} & L59 & L63-25 &  & 0(1) & 35 & 134 \\ \cline{ 3- 8}
\multicolumn{ 1}{|l|}{} & \multicolumn{ 1}{l|}{} & L63-25 & L64\_39 &  & 0 & $>$44 & 199 \\ \cline{ 3- 8}
\multicolumn{ 1}{|l|}{} & \multicolumn{ 1}{l|}{} & L64\_39 & L63-25 &  & 0 & 41 & 135 \\ \cline{ 3- 8}
\multicolumn{ 1}{|l|}{} & \multicolumn{ 1}{l|}{} & L59 & L64\_39 &  & 0 & 35 & 30.4 \\ \cline{ 3- 8}
\multicolumn{ 1}{|l|}{} & \multicolumn{ 1}{l|}{} & L64\_39 & L59 &  & 0 & 37 & 31.3 \\ \hline
\pagebreak[3]\multicolumn{ 1}{|c|}{Fib2} & \multicolumn{ 1}{c|}{A} & L64\_28 & L64\_28 & L64\_28 & 0 & 40 & 28.4 \\ \cline{ 3- 8}
\nopagebreak\multicolumn{ 1}{|l|}{} & \multicolumn{ 1}{l|}{} & L64\_39 & L64\_28 & L64\_28 & 0(2) & 40 & 27.9 \\ \cline{ 3- 8}
\nopagebreak\multicolumn{ 1}{|l|}{} & \multicolumn{ 1}{l|}{} & L64\_39 & L64\_32 & L64\_28 & 0 & 39 & 27.5 \\ \cline{ 3- 8}
\nopagebreak\multicolumn{ 1}{|l|}{} & \multicolumn{ 1}{l|}{} & L64\_28 & L64\_39 & L64\_28 & 0 & 40 & 27.3 \\ \cline{ 3- 8}
\nopagebreak\multicolumn{ 1}{|l|}{} & \multicolumn{ 1}{l|}{} & L64\_32 & L64\_39 & L64\_28 & 0 & 40 & 27.5 \\ \cline{ 3- 8}
\nopagebreak\multicolumn{ 1}{|l|}{} & \multicolumn{ 1}{l|}{} & L64\_39 & L64\_39 & L64\_28 & 0 & 40 & 27.4 \\ \cline{ 3- 8}
\nopagebreak\multicolumn{ 1}{|l|}{} & \multicolumn{ 1}{l|}{} & L64\_39 & L64\_28 & L64\_32 & 0 & 40 & 27.3 \\ \cline{ 3- 8}
\nopagebreak\multicolumn{ 1}{|l|}{} & \multicolumn{ 1}{l|}{} & L64\_28 & L64\_39 & L64\_32 & 0 & 40 & 27.9 \\ \cline{ 3- 8}
\nopagebreak\multicolumn{ 1}{|l|}{} & \multicolumn{ 1}{l|}{} & L64\_28 & L64\_28 & L64\_39 & 0(1) & 40 & 27.4 \\ \cline{ 3- 8}
\nopagebreak\multicolumn{ 1}{|l|}{} & \multicolumn{ 1}{l|}{} & L64\_32 & L64\_28 & L64\_39 & 0 & 39 & 27.7 \\ \cline{ 3- 8}
\nopagebreak\multicolumn{ 1}{|l|}{} & \multicolumn{ 1}{l|}{} & L64\_39 & L64\_28 & L64\_39 & 0 & 40 & 27.3 \\ \cline{ 3- 8}
\nopagebreak\multicolumn{ 1}{|l|}{} & \multicolumn{ 1}{l|}{} & L64\_28 & L64\_32 & L64\_39 & 0 & 40 & 27.3 \\ \cline{ 3- 8}
\nopagebreak\multicolumn{ 1}{|l|}{} & \multicolumn{ 1}{l|}{} & L64\_28 & L64\_39 & L64\_39 & 0 & 40 & 27.3 \\ \cline{ 3- 8}
\nopagebreak\multicolumn{ 1}{|l|}{} & \multicolumn{ 1}{l|}{} & L64\_39 & L64\_39 & L64\_39 & 0 & 40 & 27.4 \\ \cline{ 2- 8}
\pagebreak[3]\multicolumn{ 1}{|l|}{} & \multicolumn{ 1}{c|}{B} & L47-115 & L47-115 & L47-115 & 0(2) & $>$44 & 297.0 \\ \cline{ 3- 8}
\multicolumn{ 1}{|l|}{} & \multicolumn{ 1}{l|}{} & L63-25 & L63-25 & L63-25 & 0(2) & $>$44 & 293.0 \\ \cline{ 3- 8}
\multicolumn{ 1}{|l|}{} & \multicolumn{ 1}{l|}{} & L59 & L59 & L59 & 0(1) & 32 & 27.4 \\ \cline{ 3- 8}
\multicolumn{ 1}{|l|}{} & \multicolumn{ 1}{l|}{} & L63 & L63 & L63 & 0 & 38 & 27.3 \\ \cline{ 2- 8}
\pagebreak[2]\multicolumn{ 1}{|l|}{} & \multicolumn{ 1}{c|}{C} & L63-25 & L59 & L64\_39 & 0(1) & 39 & 113.0 \\ \cline{ 3- 8}
\multicolumn{ 1}{|l|}{} & \multicolumn{ 1}{l|}{} & L63-25 & L64\_39 & L59 & 0 & 32 & 113.0 \\ \cline{ 3- 8}
\multicolumn{ 1}{|l|}{} & \multicolumn{ 1}{l|}{} & L59 & L63-25 & L64\_39 & 0 & 38 & 81.1 \\ \cline{ 3- 8}
\multicolumn{ 1}{|l|}{} & \multicolumn{ 1}{l|}{} & L59 & L64\_39 & L63-25 & 0 & 39 & 158.3 \\ \cline{ 3- 8}
\multicolumn{ 1}{|l|}{} & \multicolumn{ 1}{l|}{} & L64\_39 & L63-25 & L59 & 0 & 31 & 81.0 \\ \cline{ 3- 8}
\multicolumn{ 1}{|l|}{} & \multicolumn{ 1}{l|}{} & L64\_39 & L59 & L63-25 & 0 & 42 & 159.0 \\ \hline
\pagebreak[3]\multicolumn{ 1}{|c|}{Trib} & \multicolumn{ 1}{c|}{A} & L64\_28 & L64\_28 & L64\_28 & 0(2) & 42 & 27.2 \\ \cline{ 3- 8}
\multicolumn{ 1}{|l|}{} & \multicolumn{ 1}{l|}{} & L64\_39 & L64\_28 & L64\_28 & 0 & 43 & 27.1 \\ \cline{ 3- 8}
\multicolumn{ 1}{|l|}{} & \multicolumn{ 1}{l|}{} & L64\_39 & L64\_32 & L64\_28 & 0(1) & 42 & 28.0 \\ \cline{ 3- 8}
\multicolumn{ 1}{|l|}{} & \multicolumn{ 1}{l|}{} & L64\_28 & L64\_39 & L64\_28 & 0(1) & 42 & 28.1 \\ \cline{ 3- 8}
\multicolumn{ 1}{|l|}{} & \multicolumn{ 1}{l|}{} & L64\_32 & L64\_39 & L64\_28 & 0 & 42 & 27.1 \\ \cline{ 3- 8}
\multicolumn{ 1}{|l|}{} & \multicolumn{ 1}{l|}{} & L64\_39 & L64\_39 & L64\_28 & 0(1) & 42 & 27.2 \\ \cline{ 3- 8}
\multicolumn{ 1}{|l|}{} & \multicolumn{ 1}{l|}{} & L64\_39 & L64\_28 & L64\_32 & 0 & 43 & 27.1 \\ \cline{ 3- 8}
\multicolumn{ 1}{|l|}{} & \multicolumn{ 1}{l|}{} & L64\_28 & L64\_39 & L64\_32 & 0(1) & 42 & 27.1 \\ \cline{ 3- 8}
\multicolumn{ 1}{|l|}{} & \multicolumn{ 1}{l|}{} & L64\_28 & L64\_28 & L64\_39 & 0 & 42 & 28.0 \\ \cline{ 3- 8}
\multicolumn{ 1}{|l|}{} & \multicolumn{ 1}{l|}{} & L64\_32 & L64\_28 & L64\_39 & 0 & 42 & 27.2 \\ \cline{ 3- 8}
\multicolumn{ 1}{|l|}{} & \multicolumn{ 1}{l|}{} & L64\_39 & L64\_28 & L64\_39 & 0(1) & 43 & 27.1 \\ \cline{ 3- 8}
\multicolumn{ 1}{|l|}{} & \multicolumn{ 1}{l|}{} & L64\_28 & L64\_32 & L64\_39 & 0 & 43 & 27.1 \\ \cline{ 3- 8}
\multicolumn{ 1}{|l|}{} & \multicolumn{ 1}{l|}{} & L64\_28 & L64\_39 & L64\_39 & 0(2) & 42 & 27.3 \\ \cline{ 3- 8}
\multicolumn{ 1}{|l|}{} & \multicolumn{ 1}{l|}{} & L64\_39 & L64\_39 & L64\_39 & 0 & 43 & 27.1 \\ \cline{ 2- 8}
\multicolumn{ 1}{|l|}{} & \multicolumn{ 1}{c|}{B} & L47-115 & L47-115 & L47-115 & 1 & $>$44 & 299.0 \\ \cline{ 3- 8}
\multicolumn{ 1}{|l|}{} & \multicolumn{ 1}{l|}{} & L63-25 & L63-25 & L63-25 & 0(1) & $>$44 & 298.0 \\ \cline{ 3- 8}
\multicolumn{ 1}{|l|}{} & \multicolumn{ 1}{l|}{} & L59 & L59 & L59 & 0 & 35 & 27.2 \\ \cline{ 3- 8}
\multicolumn{ 1}{|l|}{} & \multicolumn{ 1}{l|}{} & L63 & L63 & L63 & 0(1) & 41 & 27.2 \\ \cline{ 2- 8}
\multicolumn{ 1}{|l|}{} & \multicolumn{ 1}{c|}{C} & L63-25 & L59 & L64\_39 & 0(1) & 39 & 172.0 \\ \cline{ 3- 8}
\multicolumn{ 1}{|l|}{} & \multicolumn{ 1}{l|}{} & L63-25 & L64\_39 & L59 & 0(1) & 41 & 173.0 \\ \cline{ 3- 8}
\multicolumn{ 1}{|l|}{} & \multicolumn{ 1}{l|}{} & L59 & L63-25 & L64\_39 & 0 & 35 & 106.0 \\ \cline{ 3- 8}
\multicolumn{ 1}{|l|}{} & \multicolumn{ 1}{l|}{} & L59 & L64\_39 & L63-25 & 0 & 34 & 70.5 \\ \cline{ 3- 8}
\multicolumn{ 1}{|l|}{} & \multicolumn{ 1}{l|}{} & L64\_39 & L63-25 & L59 & 0 & 41 & 107.0 \\ \cline{ 3- 8}
\multicolumn{ 1}{|l|}{} & \multicolumn{ 1}{l|}{} & L64\_39 & L59 & L63-25 & 0(1) & 40 & 74.3 \\ \hline


\end{longtable}

We can make the following observations based on the results in statistical tests:
\begin{enumerate}
\item The quality of LCGs has improved substantially when we combined them according to infinite words with the WELLDOC property. This can be seen in the TestU01 BigCrush results. While for LCGs $13$ to $19$ tests have clearly failed (the only exception is the generator L63-25 with two failures -- see Table~\ref{table:LCG}), almost all of the BigCrush tests passed. The worst result was to have one BigCrush test failed for the Tribonacci combination and one for the Fibonacci combination of L47-115 generators. The likely reason is that the generator L47-115 has the shortest period of all tested LCGs.
\item The results of the PractRand battery confirm the above findings. For instance, in the case of LCGs with modulo $2^{64}$, the test started to find irregularities in the distribution of the least significant bit of tested PRNGs output at around $2$TB sample size. Compare it with the sample size of $8$GB to $32$GB when fast plain LCGs started to fail the test. The PractRand battery applies different filters on the input stream and all failures appeared for {\tt Low1/32} filter where only the least significant bit of the PRNG output is used. It corresponds to a known weakness of power-of-2 modulo LCGs: lower bits of the output have significantly smaller period than the LCG itself.
    The quality of the PRNGs can be therefore further improved by combining LCGs that do not show flaws for the least significant bits or by using for example just $16$ upper bits of the LCGs output.
\item The quality of the PRNG is linked to the quality of the underlying LCG. When looking at the group B in Table~\ref{table:APRNG_Table}, we observe that the PractRand results of the arising PRNGs are closely related to the succes of LCGs from Table~\ref{table:LCG} in the PractRand tests.
\item Another interesting observation is that using the instances of the same LCG (with only sufficiently distinct seeds) produces as good results as combination of different LCGs (multipliers and shifts are different, but the modulus is the same). It is just important to make sure that starting states of the LCGs are far apart enough. Refer to the group A in Table~\ref{table:APRNG_Table}.
\item The lower quality LCG dictates the quality of resulting PRNG. When mixing LCGs with different quality, use better ones as replacement for more frequent letters in the aperiodic word.
%

 Please refer to the group C in Table~\ref{table:APRNG_Table}. For example for the Fibonacci word compare first two rows in the group C - the order of LCGs is merely swapped but the difference in the sample size for which PractRand starts to fail is $8\times$. This is even more significant for the Tribonacci based generators where the difference between the worst and best PractRand results when reordering the underlying LCGs is given by factor $128\times$.
\item On the other hand, results from the group A in Table~\ref{table:APRNG_Table} demonstrate that when using generators of similar quality (same modulus, similar deficiencies), the order in which generators are used to substitute the letters of the infinite word does not influence the quality of the resulting generator.
\item We can also see that the modified Fibonacci word (see Remark~\ref{modifiedFib}) does not produce better results than the Fibonacci word. Clearly, a regular structure of $2$'s on every other position does not help to produce a~better random sequence even if we mix now three LCGs instead of two as in the case of the Fibonacci word.
\item Results for the Tribonacci word are better than for the Fibonacci word. (We have observed this fact for all ternary Arnoux-Rauzy words in comparison to Sturmian words.) It seems therefore that mixing three LCGs is better than using just two LCGs, assuming that an infinite word with the WELLDOC property is used for mixing. We expect naturally that the better chosen LCGs (or even some other modern fast linear PRNGs,  e.g. \emph{mt19937} or nonlinear PRNGs based on the AES cipher) we combine according to an infinite word with the WELLDOC property, the better their results in statistical tests will be.
\item We have also tested LCGs with $m=2^{31}-1$. It has revealed that if the underlying generators have poor statistical properties, then the PRNG will not be able to mask it. In particular, you cannot expect that PRNGs -- despite their infinite aperiodic nature -- will fix the short period problem. Once the period of the underlying LCG is exhausted, statistical tests will find irregularities in the output of the PRNG.
\end{enumerate}

In conclusion, we summarize the main results from the user point
of view:
\begin{itemize}
\item Using different instances of the same LCG to form a~new generator based on the infinite word with the WELLDOC property gives a generator with improved statistical properties.
\item The introduced method of generation of morphic words is very fast and supports parallel processing.
\item The period of underlying generators has to be large enough -- much larger than the number of needed values.
\item When using different types of the underlying LCGs to form a~PRNG, close attention has to be paid to the right
order of the combined LCGs. The generator with the worst
properties should be used to replace the least frequent letter of
the aperiodic word. Moreover, statistical properties of the
resulting PRNG are ruled by the deficiencies of the worst used
generator.
\item We have used the LCGs only for study reasons.
Instead of LCGs, the modern generators (of user choice) could be
used as underlying PRNG to obtain better results. We have done
testing with two instances (respectively three for the Tribonacci
and other Arnoux-Rauzy words) of Mersenne twister $19937$ as the
underlying generator. The newly constructed generator has passed
all the empirical tests on randomness we have executed (in
contrary to Mersenne twister $19937$ itself which is failing two
tests from TestU01's BigCrush battery). For the practical usage
Arnoux-Rauzy (AR) words are very appealing since there is an
infinite number of AR words and we have implementation in place to
create the AR words based on user input (it can be sought of as
the seed). Thus, we recommend to create new PRNGs based on one's
favorite modern PRNGs and the custom AR word.
\end{itemize}

\section{Open problems and future research}
Concerning the combinatorial part of our paper, one of the
interesting open questions there is finding large families of
infinite words satisfying the WELLDOC property. For example, which
morphic words have the WELLDOC property? Also, it seems to be
meaningful to study a weaker WELLDOC property where in
Definition~\ref{comb_cond} instead of every $m \in \mathbb N$ we
consider only a particular $m$. For instance, one can search for
words satisfying such a modified WELLDOC condition for $m=2$,
$m=2^{\ell}$ etc. Another question to be asked is how to construct
words with the WELLDOC property over larger alphabets using words
with such a property over smaller alphabets. Regarding statistical
tests, it remains to explain why PRNGs based on infinite words
with the WELLDOC property succeed in tests and to compare their
results with other comparably fast generators.

\section*{Acknowledgements}
The first author was supported by the Czech Science Foundation
grant GA\v CR 13-03538S, and thanks L'Or\'eal Czech Republic for the Fellowship Women in
Science.
The third author was partially supported by the Italian Ministry of Education (MIUR), under the PRIN 2010--11 project ``Automi e Linguaggi Formali: Aspetti Matematici e Applicativi''.
The fifth author was supported in part by
the Academy of Finland under grant 251371 and by Russian
Foundation of Basic Research (grants 12-01-00089 and 12-01-00448).


\bibliographystyle{amsplain}

\end{document}